\documentclass[11pt,twoside,a4paper]{article}
\usepackage{makeidx}
\usepackage[dvips]{graphicx}
\usepackage{amscd}
\usepackage{amsmath}
\usepackage{amssymb}
\usepackage{latexsym}
\usepackage{stmaryrd}
\usepackage[latin1]{inputenc}
\usepackage[T1]{fontenc}   
\usepackage[english]{babel}
\newcommand{\tdun}[1]{\begin{picture}(10,5)(-2,-1)
\put(0,0){\circle*{2}}
\put(3,-2){\tiny #1}
\end{picture}}

\input{xy}
\xyoption{all}

\newcommand{\g}{\mathfrak{g}}
\newcommand{\h}{\mathfrak{h}}
\newcommand{\U}{\mathcal{U}}
\newcommand{\D}{\mathcal{D}}
\newcommand{\tdelta}{\tilde{\Delta}}

\newcommand{\FQSym}{\mathbf{FQSym}}
\newcommand{\PQSym}{\mathbf{PQSym}}
\newcommand{\NCQSym}{\mathbf{NCQSym}}
\newcommand{\YSym}{\mathbf{YSym}}
\renewcommand{\H}{\mathcal{H}}

\title{Free and cofree Hopf algebras}
\date{}
\author{L. Foissy \\
\\
{\small{\it Laboratoire de Math\'ematiques, Universit\'e de Reims}}\\
\small{{\it Moulin de la Housse - BP 1039 - 51687 REIMS Cedex 2, France}}\\
\small{e-mail : loic.foissy@univ-reims.fr}}

\newtheorem{defi}{\indent Definition}
\newtheorem{lemma}[defi]{\indent Lemma}
\newtheorem{cor}[defi]{\indent Corollary}
\newtheorem{theo}[defi]{\indent Theorem}
\newtheorem{prop}[defi]{\indent Proposition}

\newenvironment{proof}{{\bf Proof.}}{\hfill $\Box$}

\begin{document}

\maketitle

ABSTRACT. We first prove that a graded, connected, free and cofree Hopf algebra is always self-dual;
then that two graded, connected, free and cofree Hopf algebras are isomorphic if, and only if, they have the same Poincaré-Hilbert formal series. 
If the characteristic of the base field is zero, we prove that the Lie algebra of the primitive elements of such an object is free, and we deduce
a characterization of the formal series of free and cofree Hopf algebras by a condition of growth of the coefficients. 
We finally show that two graded, connected, free and cofree Hopf algebras are isomorphic as (non graded) Hopf algebras
if, and only if, the Lie algebra of their primitive elements have the same number of generators.\\

KEYWORDS. Free and cofree Hopf algebras; self-duality.\\

AMS CLASSIFICATION. 16W30.

\tableofcontents

\section*{Introduction}

The theory of combinatorial Hopf algebras has known a great extension in the last decade.
It turns out that an important part of the Hopf algebras studied in this theory are both free and cofree, for example:
\begin{enumerate}
\item The Hopf algebra $\FQSym$ of free quasi-symmetric functions \cite{Malvenuto2,Malvenuto,Poirier}, also known as the Malvenuto-Reutenauer Hopf algebra.
\item The Hopf algebra $\PQSym$ of parking functions \cite{Novelli2,Novelli1}.
\item The Hopf algebra on set compositions, called $P\Pi$ in \cite{Aguiar3} and  $\NCQSym$ in \cite{Bergeron}.
\item The isomorphic Hopf algebras $R\Pi$ and $S\Pi$ on pairs of permutations of \cite{Aguiar3}.
\item The Loday-Ronco Hopf algebra $\H_{LR}$ of planar binary trees \cite{Loday2} and its dual $\YSym$ \cite{Aguiar2}.
\item The Hopf algebras of planar rooted trees \cite{Foissy2,Holtkamp}, also known as the non-commutative Connes-Kreimer Hopf algebra $\H_{NCK}$,
and its decorated versions.
\item The Hopf algebra of double posets $\H_{DP}$ \cite{Reutenauer}.
\item The Hopf algebra of ordered forests $\H_o$ and its subalgebra of heap-ordered forests $\H_{ho}$ \cite{FoissyUnterberger}.
\item The free $2$-$As$ algebras \cite{Loday}.
\item The Hopf algebra of uniform block permutations $\H_{UBP}$  \cite{Aguiar}.
\item And, if the characteristic of the base field $K$ is zero, the Hopf algebra $K[X]$.
\end{enumerate}
Note that the space of generators and the space of cogenerators are not the same in these examples, except for $K[X]$.

It also turns out that certain of these objects are self-dual. The self-duality is proved by the construction of a more or less explicit pairing for $\FQSym$, 
the Connes-Kreimer Hopf algebras, the Hopf algebra of posets and $K[X]$. In the case of $\PQSym$ or the Hopf algebra of ordered forests, 
the self-duality and the cofreeness is proved by the construction of an non-explicit isomorphism with a Connes-Kreimer Hopf algebra, 
using non-associative products and coproducts and a rigidity theorem \cite{Foissy3,Foissy4}. 
It can similarly be proved that the Hopf algebra of double posets is isomorphic to a Connes-Kreimer Hopf algebra.
As a corollary, any two of these objects with the same formal series are isomorphic,
and their Lie algebras of primitive elements are free if the characteristic of the base field is zero. 
Note that the self-duality of $R\Pi$ and $S\Pi$ and the freeness of their Lie algebras of primitive elements were not proved yet and are implied by
 theorem \ref{4} and corollary \ref{12} of the present text. \\

In this context, the following questions are natural: 
\begin{enumerate}
\item Is a free and cofree Hopf algebra always self-dual?
\item Are two free and cofree Hopf algebras with the same formal series always isomorphic?
\item What can be said of the structure of the Lie algebra of the primitive elements of a free and cofree Hopf algebra?
\end{enumerate}
We here give a positive answer to questions 1 and 2 and prove for question 3 that these Lie algebras are free if the characteristic of the base field is $0$.
We first prove that any free and cofree Hopf algebra $H$ can be given a non-degenerate, symmetric Hopf pairing, so is self-dual. 
More precisely, if $\g$ is the Lie algebra of the primitive elements of $H$ and $H^+$ is the augmentation ideal of $H$, then any non-degenerate symmetric pairing
on the space of indecomposable primitive elements $\frac{\g}{\g \cap H^{+2}}$ can be (not uniquely) extended to a non-degenerate, 
symmetric Hopf pairing on $H$ (theorem \ref{4}).
We then deduce that two free and cofree Hopf algebras $H$ and $H'$ are isomorphic as graded Hopf algebras if, and only if,
$H$ and $H'$ have the same Poincaré-Hilbert formal series (theorem \ref{7}).

Restricting us to a base field of characteristic zero, we characterize the formal series of free and cofree Hopf algebras.
We first prove that the Lie algebra $\g$ is free, answering question 3 (corollary \ref{12}). 
We deduce in proposition \ref{13} relations between the coefficients of the following formal series:
$$R(h)=\sum_{n=0}^\infty dim(H_n)h^n,\hspace{.5cm}
P(h)=\sum_{n=1}^\infty dim(\g_n) h^n,\hspace{.5cm}
S(h)=\sum_{n=1}^\infty dim\left(\left(\frac{\g}{\g \cap H^{+2}}\right)_n\right) h^n.$$
Using non-commutative Connes-Kreimer Hopf algebras, we prove that the formal series $S(h)$ can be arbitrarily chosen (corollary \ref{16}).
As a consequence, the formal series of free and cofree Hopf algebras are characterized by growth conditions of the coefficients, expressed in corollary \ref{18}.
As the growth condition characterizing the formal series of non-commutative Connes-Kreimer Hopf algebras is distinct (theorem \ref{20}),
this implies that there exists free and cofree Hopf algebras that are neither $K[X]$ nor Connes-Kreimer Hopf algebras.

Let $H$ and $H'$ be two graded, connected, free and connected Hopf algebras. Are they isomorphic as (non-graded) Hopf algebras?
In order to answer this question, we first refine their graduation into a bigraduation in the last section of this text. We then prove that 
$H$ and $H'$ are isomorphic if, and only if, $\frac{\g}{[\g,\g]}$ and $\frac{\g'}{[\g',\g']}$ have the same dimension. As a corollary, 
$\FQSym$, $\PQSym$, $\NCQSym$, $R\Pi$, $S\Pi$, $\H_{LR}$, $\YSym$, $\H_{NCK}$ and its decorated version $\H_{NCK}^\D$ 
for any non-empty graded set $\D$, $\H_{DP}$, $\H_o$, $\H_{ho}$, the free $2$-$As$ algebras, and $\H_{UBP}$ are isomorphic.\\

 {\bf Notations.} \begin{enumerate}
 \item In the whole text, $K$ is a commutative field of characteristic $\neq 2$. Any algebra, coalgebra, Hopf algebra\ldots of the text will be taken over $K$.
 \item If $H$ is a Hopf algebra, we denote by $H^+$ its augmentation ideal and by $Prim(H)$ or by $\g$ if there is no ambiguity the Lie algebra of its primitive elements.
 Moreover, $H^+$ inherits a coassociative, non counitary coproduct $\tdelta$ defined by $\tdelta(x)=\Delta(x)-x \otimes 1-1\otimes x$  for all $x \in H^+$.
 The square product $H^{+2}$ of the ideal $H^+$ by itself is called the space of {\it decomposable elements} \cite{Milnor}; 
 the quotients $\frac{H^+}{H^{+2}}$ and $\frac{\g}{\g \cap H^{+2}}$ are respectively called the space of {\it indecomposable elements}
 and of {\it indecomposable primitive elements}.
 \end{enumerate}
 
 {\bf Aknowledgements.} This work was partially supported by a PEPS.
 

\section{Free and cofree Hopf algebras are self-dual}

\subsection{Hopf pairings}

We first recall the following definition:

\begin{defi}\label{d1}\textnormal{
\begin{enumerate}
\item Let $H,H'$ be two graded, connected Hopf algebras. A {\it homogeneous Hopf pairing} is a bilinear form $\langle-,-\rangle:H\times H' \longrightarrow K$ such that:
\begin{enumerate}
\item For all $x\in H$, $x'\in H'$, $\langle x,1\rangle=\varepsilon(x)$ and $\langle 1,x'\rangle=\varepsilon(x')$.
\item For all $x,y \in H$, $y',z'\in H'$, $\langle xy,z'\rangle=\langle x \otimes y, \Delta(z')\rangle$ and $\langle x,y'z'\rangle=\langle \Delta(x),y'\otimes z'\rangle$.
\item If $x,x'$ are homogeneous of different degrees, then $\langle x,x'\rangle =0$.
\end{enumerate}
\item Let $H$ be a graded, connected Hopf algebra. We shall say that $H$ is {\it self-dual} if it can be given a non-degenerate Hopf pairing
$\langle-,-\rangle:H\times H\longrightarrow K$.
\end{enumerate}}\end{defi}

{\bf Note.} As all the Hopf pairings considered here are homogeneous, we shall simply write "Hopf pairings" for "homogeneous Hopf pairings" in this text.\\

Let $\langle-,-\rangle$ be any pairing on $H \times H'$. It is a Hopf pairing if, and only if, the following map is a morphism of graded Hopf algebras:
$$\left\{\begin{array}{rcl}
H&\longrightarrow&H'^*\\
x&\longrightarrow&\langle x,-\rangle,
\end{array}\right.$$
where $H'^*$ is the graded dual of $H'$. Moreover, if the pairing is non-degenerate, then this map is an isomorphism.
Conversely, any morphism of graded Hopf algebras $\phi:H\longrightarrow H'^*$ gives a Hopf pairing, defined by $\langle x,x'\rangle=\phi(x)(x')$.
As a consequence,  a graded, connected Hopf algebra $H$ is self-dual in the sense of definition \ref{d1} if, and only if, it is isomorphic to $H^*$ as a graded Hopf algebra;
moreover, symmetry of the pairing is equivalent to self-duality of the isomorphism. 

\begin{lemma} \label{2}
Let $H,H'$ be two graded, connected Hopf algebras, and let $\langle-,-\rangle$ be a Hopf pairing on $H\times H'$. Let us fix $n \geq 1$.
If $\langle-,-\rangle_{\mid H_k \times H'_k}$ is non-degenerate for all $k<n$, then in $H'_n$, $((H^{+2})_n)^\perp=Prim(H')_n$.
\end{lemma}

\begin{proof} Let $z\in H'_n$. For any $x \in H_k$, $y\in H_{n-k}$, $1\leq k \leq n-1$:
\begin{eqnarray*}
\langle xy,z\rangle&=&\langle x\otimes y,\Delta(z)\rangle\\
&=&\langle x \otimes y,1\otimes z+z\otimes 1+\tdelta(z)\rangle\\
&=&\varepsilon(x)\langle y,z\rangle +\varepsilon(y)\langle x,z\rangle+\langle x\otimes y,\tdelta(z)\rangle\\
&=&\langle x\otimes y,\tdelta(z)\rangle.
\end{eqnarray*}
If $z$ is primitive, then $\tdelta(z)=0$, so $\langle xy,z\rangle=0$ and $z\in ((H^{+2})_n)^\perp$. Conversely, if $z\in ((H^{+2})_n)^\perp$,
then $\tdelta(z) \in ((H^+\otimes H^+)_n)^\perp$. As the pairing is non-degenerate in degree $<n$:
$$((H^+\otimes H^+)_n)^\perp=\left(\sum_{k=1}^{n-1}H_k \otimes H_{n-k}\right)^\perp=(0),$$
so $z$ is primitive. \end{proof}\\

As a consequence, if $H$ is a  graded, connected, self-dual Hopf algebra, any non-degenerate Hopf pairing on $H$ induces a non-degenerate, homogeneous pairing 
on $\frac{\g}{\g\cap H^{+2}}$, where $\g$ is the Lie algebra $Prim(H)$. 
This pairing will be called {\it the induced pairing on } $\frac{\g}{\g\cap H^{+2}}$.

\begin{lemma}\label{3}
Let $H$ be a graded, connected, self-dual Hopf algebra, with a symmetric, non-degenerate Hopf pairing $\langle-,-\rangle$. Let us fix $n \in \mathbb{N}$.
Let $\h_n$ be a complement of $(\g \cap H^{+2})_n$ in $\g_n$ and let $m_n$ be a complement of $(\g \cap H^{+2})_n$ in $H^{+2}$.
Then $\h_n$ is non-isotropic, that is to say the restriction  $\langle-,-\rangle_{\mid \h_n\times \h_n}$ is non-degenerate.
There exists a complement $w_n$ of $\g_n+H^{+2}_n$ in $H_n$, such that, in $H_n$:
\begin{itemize}
\item $w_n^\perp=w_n \oplus \h_n \oplus m_n$ and the restriction of the pairing to $(\g\cap H^{+2})_n \times w_n$ is non-degenerate.
\item $\h_n^\perp=(\g\cap H^{+2})_n\oplus m_n \oplus w_n$ and the restriction of the pairing to $\h_n\times \h_n$ is non-degenerate.
\item $m_n^\perp=(\g\cap H^{+2})_n\oplus \h_n \oplus w_n$ and the restriction of the pairing to $m_n\times m_n$ is non-degenerate.
\item $((\g\cap H^{+2})_n)^\perp=(\g\cap H^{+2})_n\oplus m_n \oplus h_n$ and the restriction of the pairing to $w_n\times (\g\cap H^{+2})_n$ is non-degenerate.
\end{itemize}
\end{lemma}

\begin{proof} By lemma \ref{2}, $\g_n \cap \g_n^\perp=(\g\cap H^{+2})_n$ and $H^{+2}_n \cap (H^{+2}_n)^\perp=(\g\cap H^{+2})_n$.
Hence, $\h_n$ and $m_n$ are non-isotropic subspaces of $H_n$. Moreover:
$$((\g\cap H^{+2})_n)^\perp=\g_n^\perp+((H^{+2})_n)^\perp=H^{+2}_n+\g_n=(\g\cap H^{+2})_n\oplus m_n\oplus \h_n.$$

Let us choose any complement $w_n$ of $\g_n+H^{+2}_n$ in $H_n$. As the pairing is non-degenerate:
\begin{eqnarray*}
dim(\g_n)&=&dim(H_n)-dim(\g_n^\perp)\\
dim(\g_n)&=&dim(H_n)-dim(H^{+2})_n\\
dim((\g\cap H^{+2})_n)+dim(\h_n)&=&dim(H_n)-dim((\g\cap H^{+2})_n)-dim(m_n)\\
dim((\g\cap H^{+2})_n)+dim(\h_n)&=&dim(\h_n)+dim(w_n),
\end{eqnarray*}
so $dim(w_n)=dim((\g\cap H^{+2})_n)$. Moreover, $(\g \cap H^{+2})_n^\perp=\g_n \oplus (H^{+2})_n$  by lemma \ref{2}, 
so $w_n\cap ((\g \cap H^{+2})_n)^\perp=(0)$ by choice of $w_n$. Hence, the restriction of the pairing to $(\g\cap H^{+2})_n \times w_n$ is non-degenerate. 
We finally obtain a decomposition:
$$H_n=(\g \cap H^{+2})_n \oplus m_n \oplus \h_n \oplus w_n.$$
Let us choose an adapted basis of $H_n$: $(x_i)_{i\in I} \cup (y_j)_{j\in J}\cup (z_k)_{k\in K} \cup (t_i)_{i\in I}$
(we can choose the same set of indices for the bases of $(\g\cap H^{+2})_n$ and $w_n$, as they have the same dimension).
In this basis, the matrix of the pairing has the following form:
$$\left(\begin{array}{cccc}
0&0&0&C\\
0&A&0&D\\
0&0&B&E\\
C^T&D^T&E^T&F
\end{array}\right),$$
where $A,B$ are symmetric, invertible matrices, and $C$ is an invertible matrix.
Changing the basis of $w_n$ so that $(x_i)_{i\in I}$ and $(t_i)_{i\in I}$ are dual, we can assume that $C$ is an identity matrix.
Let $A_j$ be the $j$-th column of $A$, $B_k$ the $k$-th column of $B$, and so on. As $A$ and $B$ are invertible, there exists scalars $\lambda^{(i)}_j$ 
and $\mu^{(i)}_k$ such that:
$$D_i=\sum_{j\in J} \lambda^{(i)}_j A_j,\hspace{.5cm} E_i=\sum_{k\in K}\mu^{(i)}_k B_k.$$
We then put:
$$t'_i=t_i-\sum_{j\in J} \lambda^{(i)}_j y_j -\sum_{k\in K} \mu^{(i)}_k z_k-\sum_{i'\in I} \frac{f_{i,i'}}{2} x_{i'}.$$
An easy computation proves that $t'_i$ is orthogonal to $y_j$, $z_k$, and $t'_{i'}$ for all $j,k,i'$. Taking the vector space $w'_n$ generated by
the $t'_i$, we obtain a complement of $\g_n+H^{+2}_n$ such that in an adapted basis of $H_n=(\g\cap H^{+2})_n \oplus m_n \oplus \h_n \oplus w'_n$,
the matrix of the pairing has the following form:
$$\left(\begin{array}{cccc}
0&0&0&C\\
0&A&0&0\\
0&0&B&0\\
C^T&0&0&0
\end{array}\right),$$
where $A$ and $B$ are symmetric, invertible matrices and $C$ is an invertible matrix.
The assertions on the orthogonals are then immediate.\end{proof}\\

{\bf Remark.} As a consequence, choosing bases of $(\g \cap H^{+2})_n$ and $w_n$ in duality, the matrix of the pairing has the following form:
$$\left(\begin{array}{cccc}
0&0&0&I\\
0&A&0&0\\
0&0&B&0\\
I&0&0&0
\end{array}\right),$$
where $A$ and $B$ are symmetric, invertible matrices and $I$ is an identity matrix.

\subsection{Self-duality of a free and cofree Hopf algebra}

\begin{lemma}\label{lemmeplus}
Let $H$ be a graded, connected, free and cofree Hopf algebra. Then, for all $n\geq 1$:
$$dim(\g_n)=dim\left(\left(\frac{H^+}{H^{+2}}\right)_n\right).$$
\end{lemma}

\begin{proof} As $H$ is free, there exists a graded subspace $V\subseteq H$, such that $H\approx T(V)$ as an algebra.
We shall consider the following formal series:
$$R(h)=\sum_{k=0}^\infty dim(H_k)h^k,\hspace{.5cm} P(h)=\sum_{k=0}^\infty dim(\g_k)h^k,\hspace{.5cm} 
G(h)=\sum_{k=0}^\infty dim(V_k)h^k.$$
As $H\approx T(V)$ as an algebra, $H^+=V\oplus H^{+2}$, so for all $n \geq 1$, $dim(V_n)=dim\left(\left(\frac{H^+}{H^{+2}}\right)_n\right)$.
Moreover $R(h)=\frac{1}{1-G(h)}$ or equivalently $G(h)=1-\frac{1}{R(h)}$.

As $H$ is cofree, it is isomorphic as a coalgebra to $coT(\g)$, the tensor coalgebra cogenerated by $\g$ (that is to say $T(\g) $ as a vector space,
with the deconcatenation product). So $R(h)=\frac{1}{1-P(h)}$ or equivalently $P(h)=1-\frac{1}{R(h)}=G(h)$.
Hence, for all $n \geq 1$, $dim(\g_n)=dim(V_n)$. \end{proof}

\begin{theo} \label{4}
Let $H$ be a graded, connected, free and cofree Hopf algebra. Let us choose a non-degenerate, homogeneous pairing on 
$\frac{\g}{\g\cap H^{+2}}$. There exists a non-degenerate Hopf pairing $\langle-,-\rangle$ on $H$, inducing the chosen
pairing on $\frac{\g}{\g\cap H^{+2}}$. Moreover, if the pairing on $\frac{\g}{\g\cap H^{+2}}$ is symmetric, then we can assume that the pairing 
on $H$ is symmetric.
\end{theo}

\begin{proof} For all $n \geq 1$, let us choose a complement $m_n$ of $(H^{+2}\cap \g)_n$ in $(H^{+2})_n$, 
a complement $\h_n$ of $(H^{+2}\cap \g)_n$ in $\g_n$, and a complement $w_n$ of $(H^{+2}+\g)_n$ in $H_n$. Hence:
$$H_n=(H^{+2}\cap \g)_n\oplus m_n \oplus \h_n \oplus w_n.$$
Note that $\h_n$ is isomorphic, as a vector space, with $\left(\frac{\g}{\g \cap H^{+2}}\right)_n$. 
Hence, there is a pairing $\langle-,-\rangle$ on $\h_n$, making the restriction to $\h_n$ of the canonical projection onto 
$\left(\frac{\g}{\g \cap H^{+2}}\right)_n$ an isometry.\\
We put $V_n=\h_n\oplus w_n$; for all $n \geq 1$, $V_n$ is a complement of $(H^{+2})_n$ in $H_n$.
We denote by $H_{\langle n\rangle}$ the subalgebra of $H$ generated by $H_0\oplus\ldots \oplus H_n$. 
Then  $H_{\langle n\rangle}$ is freely generated by $V_1\oplus\ldots \oplus V_n$. Moreover, for all $k\leq n$:
$$\Delta(H_k) \subseteq \sum_{i=0}^k H_i \otimes H_{k-i}\subseteq H_{\langle n\rangle}\otimes H_{\langle n\rangle},$$
so $H_{\langle n\rangle}$ is a Hopf subalgebra of $H$. \\

We construct by induction on $n$ a homogeneous injective Hopf algebra morphism $\phi^{(n)}$ from $H_{\langle n \rangle}$ to $H^*$ such that:
\begin{enumerate}
\item for all $x,y \in \h_n$, $\phi^{(n)} (x)(y)=\langle x,y \rangle$.
\item $\phi^{(n)}_{\mid H_{\langle n-1 \rangle}}=\phi^{(n-1)}$.
\end{enumerate}
We first define $\phi^{(0)}$ by $\phi^{(0)}(1)=\varepsilon=1_{H^*}$.
Let us assume that $\phi^{(n-1)}$ is constructed. 

\begin{itemize}
\item for $v\in V_i$, $i\leq n-1$, $\phi^{(n)}(v)=\phi^{(n-1)}(v)$. This is an element of $H_i^*$, by homogeneity of $\phi^{(n-1)}$.
\item for $v\in \h_n$, $\phi^{(n)}(v) \in H_n^*$ is defined by:
$$\phi^{(n)}(v)(x)=\left\{\begin{array}{l}
\langle  v, x\rangle\mbox{ if }x\in \h_n,\\
0 \mbox{ if }x \in H^{+2},\\
0 \mbox{ if }x \in w_n,
\end{array}\right.$$ 
where $\langle-,-\rangle$ is the pairing on $\h_n$ defined earlier.
\item for $v\in w_n$, we define $\phi^{(n)}(v) \in H_n^*$ by 
$$\phi^{(n)}(v)(x)=\left\{\begin{array}{l}
0\mbox{ if }x \in \h_n,\\
0 \mbox{ if }w \in w_n,\\
\left( (\phi^{(n-1)})^{\otimes k}\circ \tdelta^{(k-1)}(v)\right)(v_1\otimes \ldots \otimes v_k)\mbox{ if }x=v_1\ldots v_k,
\end{array}\right.$$
where $v_1,\ldots,v_k$ are homogeneous elements of $V_1\oplus \ldots \oplus V_{n-1}$, $k\geq 2$.
As $H^{+2}_n=(H_{\langle n-1\rangle})_n=T(V_1\oplus \ldots \oplus V_{n-1})_n$, this perfectly defines $\phi(v)$.
\end{itemize}
As $H_{\langle n\rangle}$ is freely generated by $V_1\oplus \ldots \oplus V_n$, we extend $\phi^{(n)}$ to an algebra morphism 
from $H_{\langle n\rangle}$ to $H^*$. As $\phi^{(n)}$ sends an element of $V_i$ to an element of $H_i^*$ for all $i\leq n$, 
$\phi^{(n)}$ is homogeneous. It clearly statisfies the two points of the induction hypothesis. It remains to prove that it is an injective Hopf algebra morphism.
Let us first prove that $\phi^{(n)}$ is a Hopf algebra morphism. It is enough to prove that for $v \in V_1\oplus \ldots \oplus V_n$,
$\left(\phi^{(n)}\otimes \phi^{(n)}\right)\circ \Delta(v)=\Delta \circ \phi^{(n)}(v)$. Using the induction hypothesis, we can restrict ourselves
to $v \in V_n=\h_n \oplus w_n$, and finally to $v \in \h_n$ or $v \in w_n$. If $v \in \h_n$, then, by definition, $\phi^{(n)}(v)$ is orthogonal to $H^{+2}$, 
so $\phi^{(n)}(v)$ is primitive by lemma \ref{2}, with $H'=H^*$. 
As $v$ is also primitive, the required assertion is proved in this case. If $v \in w_n$, let $x=v_1\ldots v_k$ and $y=v_{k+1}\ldots v_{k+l}$,
where  $v_1,\ldots,v_{k+l}$ are homogeneous elements of $V_1\oplus \ldots \oplus V_{n-1}$. Then:
\begin{eqnarray*}
\left(\tdelta\circ \phi^{(n)}(v)\right)(x \otimes y)&=&\phi^{(n)}(v)(xy)\\
&=&\left( (\phi^{(n-1)})^{\otimes k+l}\circ \tdelta^{(k+l-1)}(v)\right)(v_1\otimes \ldots \otimes v_{k+l})\\
&=&\left( (\phi^{(n-1)})^{\otimes k+l}\circ\left( \tdelta^{(k-1)}\otimes \tdelta^{(l-1)}\right)
\circ \tdelta(v)\right)(v_1\otimes \ldots \otimes v_{k+l})\\
&=&\left(\left(\left(\tdelta^{(k-1)} \circ \phi^{(n-1)}\right)\otimes \left(\tdelta^{(l-1)}\circ \phi^{(n-1)}\right)\right)
\circ \tdelta(v)\right)(v_1\otimes \ldots \otimes v_{k+l})\\
 &=&\left(\left(\phi^{(n-1)}\otimes \phi^{(n-1)}\right)\circ \tdelta(v)\right)(x \otimes y)\\
 &=&\left(\left(\phi^{(n)}\otimes \phi^{(n)}\right)\circ \tdelta(v)\right)(x \otimes y).
\end{eqnarray*}
We used the induction hypothesis for the fourth equality. This proves the required assertion, so $\phi^{(n)}$ is a Hopf algebra morphism.\\

Let us now prove the injectivity of $\phi^{(n)}$. Let $x \in V_n$, such that $\phi^{(n)}(x) \in {H^*}^{+2}$.
Let us prove that $x=0$. By homogeneity, $\phi^{(n)}(x) \in \left({H^*}^{+2}\right)_n$.
As $\phi^{(n-1)}$ is injective, comparing the dimension of the homogeneous component of degree $i$ of $H_{\langle n-1 \rangle}$
and $H^*$ for all $i\leq n-1$, we deduce that $H^*_i \subseteq Im(\phi^{(n-1)})$ if $i\leq n-1$. 
So $ \left({H^*}^{+2}\right)_n \subseteq Im(\phi^{(n-1)})$.
Hence, there exists $y\in H_{\langle n-1 \rangle}$, homogeneous of degree $n-1$, such that $\phi^{(n)}(x)=\phi^{(n-1)}(y)$.
So $y \in (H_{\langle n \rangle})^{+2}$ and $x-y \in Ker(\phi^{(n)})$.  By the induction hypothesis, $\phi^{(n)}$ restricted
to the homogeneous components of degree $\leq n-1$ is injective, so if $x-y\neq 0$, it is a non-zero element of $Ker(\phi^{(n)})$ of minimal degree:
as $\phi^{(n)}$ is a Hopf algebra morphism, $x-y$ is primitive. So $x-y \in \h_n \oplus (\g \cap H^{+2})_n$.
Moreover, $y\in (H^{+2})_n$, so $x \in (\h_n \oplus (\g \cap H^{+2})_n \oplus m_n)\cap (\h_n \oplus w_n)=\h_n$.
As the pairing on $\h_n \times \h_n$ is non-degenerate, if $x \neq 0$, there exists $z \in \h_n$, such that 
$\phi^{(n)}(x)(z)=\langle x,z\rangle\neq 0$, so $\phi^{(n)}(x) \notin \g^\perp={H^*}^{+2}$: contradiction. So $x=0$.

This proves that $\phi^{(n)}(V_n) \cap  {H^*}^{+2}=(0)$ and $\phi^{(n)}_{\mid V_n}$ is injective.
As $\phi^{(n-1)}$ is an injective Hopf algebra morphism, we deduce that $\phi^{(n)}(V_k)\cap {H^*}^{+2}=(0)$ 
and $\phi^{(n)}_{\mid V_k}$ is injective for all $k\leq n$. As $H$ is cofree, $H^*$ is free, so $\phi^{(n)}(V_1\oplus \ldots \oplus V_n)$
generates a free subalgebra of $H^*$. As $\phi^{(n)}_{\mid V_1\oplus \ldots \oplus V_n}$ is injective and
$H_{\langle n \rangle}$ is freely generated by $V_1\oplus \ldots \oplus V_n$,  $\phi^{(n)}$ is injective.\\

{\it Conclusion}. Let $x \in H$. We put $\phi(x)=\phi^{(n)}(x)$ for any $n\geq deg(x)$. By the second point of the induction, this is well-defined.
As $\phi^{(n)}$ is an injective and homogeneous Hopf algebra morphism for all $n$, $\phi$ also is. Comparing the formal series of $H$ and $H^*$,
$\phi$ is an isomorphism. So it defines a non-degenerate, homogeneous Hopf pairing on $H$. By the first point of the induction, it implies 
the chosen pairing on $\frac{\g}{\g\cap H^{+2}}$.\\

Let us finally prove the symmetry of this pairing if the pairing on $\frac{\g}{\g\cap H^{+2}}$ is symmetric.
Let $x,y\in H_n$, let us prove that $\langle x,y\rangle=\langle y,x \rangle$ by induction on $n$. 
As $H_0=K$, it is obvious if $n=0$. If $n \geq 1$, then $\langle x,y\rangle=\phi^{(n)}(x)(y)$.
It is then enough to prove this for $x,y \in \h_n$, or $w_n$, or $(H^{+2})_n$.

{\it First case.} If $x \in (H^{+2})_n$, let us put $x=x_1x_2$, where $x_1,x_2$ are homogeneous, of degree $<n$. Then, by the induction hypothesis:
$$\langle x,y\rangle=\langle x_1\otimes x_2,\Delta(y)\rangle=\langle x_1\otimes x_2,\tdelta(y)\rangle=\langle \tdelta(y),x_1\otimes x_2\rangle
=\langle y,x\rangle.$$
The proof is similar if $y\in (H^{+2})_n$.

{\it Second subcase.} If $x\in w_n$ and $y\in \h_n$ or $w_n$, then by definition of $\phi^{(n)}$,  $\langle x,y\rangle=\langle y,x \rangle=0$.
The proof is similar if $y\in w_n$ and $x\in \h_n$ or $w_n$.

{\it Last subcase.} If $x,y \in \h_n$, then $\langle x,y\rangle=\langle y,x \rangle$ as the pairing on $\frac{\g}{\g\cap H^{+2}}$ is symmetric. \end{proof}\\

{\bf Remark.} It is possible to extend a symmetric pairing on  $\frac{\g}{\g\cap H^{+2}}$  to a non-symmetric, non-degenerate Hopf pairing on $H$:
it is enough to arbitrarily change the values of $\phi^{(n)}(v)(x)$ for $v\in \h_n$ and $x\in w_n$, or $v\in w_n$ and $x\in \h_n$. 

\begin{cor} \label{5}
Let $H$ be a graded, connected Hopf algebra, free and cofree. Then it is self-dual.
\end{cor}

\begin{proof} It is enough to choose a non-degenerate, homogeneous pairing on $\frac{\g}{\g\cap H^{+2}}$ and then to apply theorem \ref{4}. \end{proof}

\subsection{Isomorphisms of free and cofree Hopf algebras}

\begin{prop}\label{6}
Let $H$ and $H'$ be two free and cofree Hopf algebras, both with a symmetric, non-degenerate Hopf pairing,  such that $\left(\frac{\g}{\g\cap H^{+2}}\right)_n$ 
and $\left(\frac{\g'}{\g'\cap {H'}^{+2}}\right)_n$, endowed with the pairings induced from the Hopf pairings, are isometric for all $n\geq 1$. 
Then there exists a Hopf algebra isomorphism from $H$ to $H'$ which is an isometry.
\end{prop}

\begin{proof} We inductively construct a homogeneous isomorphism $\phi^{(n)}:H_{\langle n\rangle} \longrightarrow H_{\langle n\rangle}'$, such that:
\begin{enumerate}
\item $\phi^{(n)}_{\mid H_{\langle n-1\rangle}}=\phi^{(n-1)}$ if $n \geq 1$.
\item $\langle\phi^{(n)}(x),\phi^{(n)}(y)\rangle=\langle x,y \rangle$ for all $x,y\in H_{\langle n\rangle}$.
\end{enumerate}
We define $\phi^{(0)}:H_{\langle 0\rangle}=K\longrightarrow H'_{\langle 0\rangle}$ by $\phi^{(0)}(1)=1$.
Let us assume that $\phi^{(n-1)}$ is constructed.
We choose a decomposition $H_n=(\g \cap H^{+2})_n \oplus m_n \oplus \h_n \oplus w_n$ as in lemma \ref{3}.
As $\phi^{(n-1)}$ is a Hopf algebra isomorphism, $\phi^{(n-1)}((H_{\langle n-1\rangle})_n)=\phi^{(n-1)}(H^{+2})_n=(H'_{\langle n-1\rangle})_n={H'}^{+2}_n$.
As a consequence, $\phi(m_n)=m'_n$ is a complement of $(\g' \cap {H'}^{+2})_n$ in ${H'}^{+2}_n$.
Moreover, as the isometry $\phi^{(n-1)}$ sends $(H_{\langle n-1 \rangle})_n=H^{+2}_n$ to $(H'_{\langle n-1\rangle})_n={H'}^{+2}_n$,
it induces an isometry from $(\g \cap H^{+2})_n$ to $(\g'\cap {H'}^{+2})_n$ and from $m_n$ to $m'_n$.

Let us choose a complement $\h_n$ of $(\g \cap H^{+2})_n$ in $\g_n$ and a complement $\h'_n$ of $(\g' \cap {H'}^{+2})_n$ in $\g'_n$.
As $\h_n$ is isometric with $\left(\frac{\g}{\g\cap H^{+2}}\right)_n$ and $\h'_n$ with $\left(\frac{\g'}{\g'\cap {H'}^{+2}}\right)_n$,
there exists an isometry $\phi:\h_n \longrightarrow \h'_n$.

Using lemma \ref{3}, there exists a basis $(x_i)_{i\in I}\cup (y_j)_{j\in J} \cup (z_k)_{k\in K} \cup (t_i)_{i\in I}$ adapted to the decomposition
$H_n=(\g\cap H^{+2})_n \oplus m_n \oplus \h_n \oplus w_n$, and a basis 
$(\phi^{(n-1)}(x_i))_{i\in I}\cup (\phi^{(n-1)}(y_j))_{j\in J} \cup (\phi(z_k))_{k\in K} \cup (t'_i)_{i\in I}$, such that the matrices of the pairing of
$H_n$ and $H'_n$ in these bases are both equal to:
$$\left(\begin{array}{cccc}
0&0&0&I\\
0&A&0&0\\
0&0&B&0\\
I&0&0&0
\end{array}\right).$$
We then define $\phi^{(n)}$ on $H_0\oplus \ldots \oplus H_n$ by putting $\phi^{(n)}(x)=x$ if $x \in H_{\langle n-1\rangle}$, 
$\phi^{(n)}(x)=\phi(x)$ if $x \in \h_n$ and $\phi^{(n)}(t_i)=t'_i$ for all $i\in I$. As $H_{\langle n\rangle}$ is freely generated
by $\h_1\oplus w_1\oplus \ldots \oplus \h_n \oplus w_n$, $\phi^{(n)}$ is extended to an algebra morphism from $H_{\langle n\rangle}$
to $H'_{\langle n\rangle}$. As its image contains $\h'_1\oplus w'_1\oplus \ldots \oplus \h'_n \oplus w'_n$, which freely generates $H'_{\langle n\rangle}$,
it is surjective. As $\phi^{(n)}$ induces a linear isomorphism from $\h_1\oplus w_1\oplus \ldots \oplus \h_n \oplus w_n$ to
$\h'_1\oplus w'_1\oplus \ldots \oplus \h'_n \oplus w'_n$, $\phi^{(n)}$ is an isomorphism. By construction (for $i=n$) and by the induction hypothesis (for $i<n$),
$\phi^{(n)}_{\mid H_i}$ is an isometry from $H_i$ to $H'_i$ for all $i\leq n$.\\

Let us prove that $\phi^{(n)}$ is a Hopf algebra isomorphism. It is enough to prove that $\Delta \circ \phi^{(n)}(x)=(\phi^{(n)} \otimes \phi^{(n)})\circ \Delta(x)$
for $x \in \h_1\oplus w_1\oplus \ldots \oplus \h_n \oplus w_n$. If $x \in \h_1\oplus w_1\oplus \ldots \oplus \h_{n-1} \oplus w_{n-1}$, 
this comes from the induction hypothesis. If $x \in \h_n$, then both $x$ and $\phi^{(n)}(x) \in \h'_n$ are primitive, so the result is obvious. 
Let us assume that $x \in w_n$. Let us take $y' \in H'_k$, $z'\in H'_l$, with $k+l=n$. As $\phi^{(n)}(H_i)=H'_i$ if $i\leq n$, 
we put $y'=\phi^{(n)}(y)$ and $z'=\phi^{(n)}(z)$:
\begin{eqnarray*}
\langle (\phi^{(n)}\otimes \phi^{(n)})\circ \Delta(x),y'\otimes z'\rangle&=&\langle (\phi^{(n)} \otimes \phi^{(n)})\circ \Delta(x),\phi^{(n)}(y)\otimes \phi^{(n)}(z)\rangle\\
&=&\langle\Delta(x),y \otimes z\rangle\\
&=&\langle x,yz\rangle\\
&=&\langle \phi^{(n)}(x),\phi^{(n)}(yz)\rangle\\
&=&\langle \phi^{(n)}(x),\phi^{(n)}(y)\phi^{(n)}(z)\rangle\\
&=&\langle \Delta \circ \phi^{(n)}(x),\phi^{(n)}(y)\otimes \phi^{(n)}(z)\rangle\\
&=&\langle \Delta \circ \phi^{(n)}(x),y'\otimes z'\rangle.
\end{eqnarray*}
By homogeneity, we deduce that $(\phi^{(n)}\otimes \phi^{(n)})\circ \Delta(x)-\Delta \circ \phi^{(n)}(x) \in (H' \otimes H')^\perp=(0)$, as the pairing of $H'$ is non-degenerate.\\

It remains to show that $\phi^{(n)}$ is an isometry: let us prove that $\langle x,y\rangle=\langle \phi^{(n)}(x),\phi^{(n)}(y)\rangle$, for $x,y \in H_{\langle n\rangle}$.
We can assume that $x$ is homogenous of degree $k$. If $k\leq n$, by homogeneity of $\phi^{(n)}$ and the pairing, this is true as 
$\phi^{(n)}_{\mid H_n}$ is an isometry from $H_n$ to $H'_n$. If $k>n$, as $H_{\langle n\rangle}$ us generated by elements of degree $\leq n$,
we can assume that $x=x_1\ldots x_k$, with $x_i$ homogeneous of degree $\leq n$ for all $i$. Then:
\begin{eqnarray*}
\langle \phi^{(n)}(x),\phi^{(n)}(y)\rangle&=&\langle \phi^{(n)}(x_1)\ldots \phi^{(n)}(x_k), \phi^{(n)}(y)\rangle\\
&=&\langle \phi^{(n)}(x_1)\otimes \ldots \otimes \phi^{(n)}(x_k), \Delta^{(k-1)}\circ \phi^{(n)}(y)\rangle\\
&=&\langle \phi^{(n)}(x_1)\otimes \ldots \otimes \phi^{(n)}(x_k), (\phi^{(n)}\otimes \ldots \otimes \phi^{(n)}) \circ \Delta^{(k-1)}(y)\rangle\\
&=&\langle x_1\otimes \ldots \otimes x_k,\Delta^{(k-1)}(y)\rangle\\
&=&\langle x_1\ldots x_k,y\rangle\\
&=&\langle x,y\rangle.
\end{eqnarray*}

{\it Conclusion.} We define $\phi:H \longrightarrow H'$ by $\phi(x)=\phi^{(n)}(x)$ for all $x\in H_{\langle n\rangle}$. By the first point of the induction,
this does not depend of the choice of $n$. Moreover, $\phi$ is clearly an isomorphism of Hopf algebras and an isometry. \end{proof}\\

{\bf Remark.} Note that the pairings used in proposition \ref{6} are any Hopf pairings, not necessarily the Hopf pairings obtained by extension of a non-degenerate
bilinear form on $\frac{\g}{\g\cap H^{+2}}$ in theorem \ref{4}.

\begin{theo} \label{7}
Let $H$ and $H'$ be two connected, graded Hopf algebras, both free and cofree. The following conditions are equivalent:
\begin{enumerate}
\item $H$ and $H'$ are isomorphic graded Hopf algebras.
\item $\frac{\g}{\g \cap H^{+2}}$ and $\frac{\g'}{\g' \cap H'^{+2}}$ are isomorphic graded spaces.
\item $H$ and $H'$ have the same Poincaré-Hilbert formal series.
\end{enumerate}
\end{theo}

\begin{proof} Clearly, $1 \Longrightarrow 2,3$.\\

$2 \Longrightarrow 1$. We choose non-degenerate, symmetric, isometric pairings on $\left(\frac{\g}{\g \cap H^{+2}}\right)_n$
 and $\left(\frac{\g'}{\g' \cap H'^{+2}}\right)_n$  for all $n \geq 1$. 
By theorem \ref{4}, there exists symmetric, non-degenerate Hopf pairings on $H$ and $H'$,  and inducing the chosen
pairings on $\frac{\g}{\g \cap H^{+2}}$ and $\frac{\g'}{\g' \cap H'^{+2}}$. By proposition \ref{6}, $H$ and $H'$ are isomorphic.\\

$3 \Longrightarrow 2$. We proceed by contraposition: let us assume that $\frac{\g}{\g \cap H^{+2}}$ and $\frac{\g'}{\g' \cap H'^{+2}}$ 
are not isomorphic graded spaces. There exists an integer $n$, such that $\left(\frac{\g}{\g \cap H^{+2}}\right)_i$ and $\left(\frac{\g'}{\g' \cap H'^{+2}}\right)_i$ 
are isomorphic spaces if $i<n$ and $\left(\frac{\g}{\g \cap H^{+2}}\right)_n$ and $\left(\frac{\g'}{\g' \cap H'^{+2}}\right)_n$ are not isomorphic spaces.
We choose non-degenerate isometric pairings on $\left(\frac{\g}{\g \cap H^{+2}}\right)_i$ and $\left(\frac{\g'}{\g' \cap H'^{+2}}\right)_i$ for all $i<n$.
From the proof of theorem \ref{4}, we can extend them to pairings on $H_{\langle n-1\rangle}$ and $H'_{\langle n-1 \rangle}$. From the proof of 
proposition \ref{6}, $H_{\langle n-1\rangle}$ and $H'_{\langle n-1 \rangle}$ are isomorphic Hopf algebras. As a consequence:
\begin{eqnarray*}
dim(H^{+2}_n)\:=\: dim((H_{\langle n-1 \rangle})_n)&=&dim((H_{\langle n-1 \rangle})_n)\:=\: dim(H'^{+2}_n),\\
dim(\g\cap H^{+2}_n)\:=\: dim((\g \cap H_{\langle n-1 \rangle})_n)&=&dim((\g' \cap H_{\langle n-1 \rangle})_n)\:=\: dim(\g' \cap H'^{+2}_n).
\end{eqnarray*}
Using a decomposition of lemma \ref{3} for $H$, we deduce:
\begin{eqnarray*}
dim(H_n)&=&dim(H^{+2}_n)+dim(\h_n)+dim(w_n)\\
&=&dim(H^{+2}_n)+dim\left(\left(\frac{g}{\g \cap H^{+2}}\right)_n \right)+dim((\g \cap H^{+2})_n).
\end{eqnarray*}
Using a decomposition of lemma \ref{3} for $H'$ and combining the different equalities, we obtain:
$$dim(H_n)-dim(H'_n)=dim\left(\left(\frac{g}{\g \cap H^{+2}}\right)_n \right)-dim\left(\left(\frac{g'}{\g' \cap H'^{+2}}\right)_n \right).$$
So this is not zero: $H$ and $H'$ do not have the same formal series. \end{proof} \\

{\bf Remark.} This result immediately implies that $\FQSym$ and $\H_{ho}$ are isomorphic, proving again a result of \cite{FoissyUnterberger}.
Similarly, $\PQSym$ and $\H_o$ are isomorphic, as proved in \cite{Foissy4}; $R\Pi$ and $S\Pi$ are isomorphic, as proved in \cite{Aguiar3}.

\section{Formal series of a free and cofree Hopf algebra in characteristic zero}

In all this section, we assume that the characteristic of the base field is $0$.

\subsection{General preliminary results}

We first put here together several technical results. \\

{\bf Notations}. Let $A$ be a graded connected Hopf algebra.
\begin{itemize}
\item $I_A$ is the ideal of $A$ generated by the commutators of $A$. The quotient $A_{ab}=A/I_A$ is a graded, connected, commutative Hopf algebra.
\item $C_A$ is the greatest cocommutative subcoalgebra of $A$. It is a graded, connected, cocommutative Hopf subalgebra of $A$.
\item $Prim(A)$ is the Lie algebra of the primitive elements of $A$.
\item $coPrim(A)$ is the Lie coalgebra of $A$, that is to say $\frac{A^+}{A^{+2}}$. The Lie cobracket is given by
$\delta(\overline{x})=(\pi \otimes \pi)\circ(\tdelta-\tdelta^{op})(x)$, where $\pi$ is the canonical projection from $A^+$ to $coPrim(A)$.
Note that $coPrim(A)$ is the space of indecomposable elements, denoted by $Q(A)$ in \cite{Milnor}.
\end{itemize}

Let $Prim(A)_{ab}=\frac{Prim(A)}{[Prim(A),Prim(A)]}$. As the Lie algebra of $Prim(A_{ab})$ is abelian, there exists a natural map:
$$\pi_A:\left\{\begin{array}{rcl}
Prim(A)_{ab}&\longrightarrow&Prim(A_{ab})\\
x+[Prim(A),Prim(A)]&\longrightarrow& x+I_A.
\end{array}\right.$$
The kernel of $\pi_A$ is $\frac{I_A \cap Prim(A)}{[Prim(A),Prim(A)]}$.\\

Let $coPrim(A)_{ab}=\{\overline{x} \in coPrim(A)\mid \delta(\overline{x})=0\}$. As $C_A$ is cocommutative, the Lie coalgebra of $C_A$ is trivial;
hence, there is a natural map:
$$\iota_A:\left\{\begin{array}{rcl}
coPrim(C_A)&\longrightarrow&coPrim(A)_{ab}\\
x+C_A^{+2}&\longrightarrow&x+A^{+2}.
\end{array}\right.$$
This map is well-defined, as $C_A^{+2} \subseteq A^{+2}$. Its kernel is $\frac{C_A \cap A^{+2}}{C_A^{+2}}$.\\

Let $A^*$ be the graded dual of $A$. Using the duality between $A$ and $A^*$, it is easy to prove:
\begin{itemize}
\item $I_A^\perp=C_{A^*}$ so $A_{ab}=C_{A^*}$ and $C_A^*=(A^*)_{ab}$.
\item $Prim(A)^\perp=(1)+{A^*}^{+2}$, so $Prim(A)^*=coPrim(A^*)$.
As a consequence of these two points, $coPrim(C_A)^*=Prim(C_A^*)=Prim((A^*)_{ab})$.
\item $[Prim(A),Prim(A)]^\perp=coPrim(A^*)_{ab}$, so $(coPrim(A)_{ab})^*=Prim(A^*)_{ab}$.
\end{itemize}
Via these identifications, $\pi_A^*=\iota_{A^*}$ and $\iota_A^*=\pi_{A^*}$. So:

\begin{lemma} \label{8}
\begin{enumerate}
\item Let us fix an integer $n$. The following assertions are equivalent:
\begin{enumerate}
\item $\pi_A$ is injective in degree $n$.
\item $\iota_{A^*}$ is surjective in degree $n$.
\item $[Prim(A),Prim(A)]_n=(Prim(A) \cap I_A)_n$.
\end{enumerate}
\item Let us fix an integer $n$. The following assertions are equivalent:
\begin{enumerate}
\item $\iota_A$ is injective in degree $n$.
\item $\pi_{A^*}$ is surjective in degree $n$.
\item $(C_A\cap A^{+2})_n=(C_A^{+2})_n$.
\end{enumerate}
\end{enumerate}
\end{lemma}

\begin{lemma}
($char(K)=0$). If $A$ is cocommutative or commutative, then $\pi_A$ and $\iota_A$ are isomorphisms.
\end{lemma}

\begin{proof} Let us assume that $A$ is cocommutative. As the characteristic of the base field is zero, by the Cartier-Quillen-Milnor-Moore theorem,
$A$ is the enveloping algebra of $Prim(A)$. Using the universal property of enveloping algebras, 
it is not difficult to prove that $A_{ab}=\U(Prim(A)_{ab})=S(Prim(A)_{ab}))$. As $Prim(\U(Prim(A_{ab})))=Prim(A_{ab})$, $\pi_A$ is clearly bijective.
As $A$ is cocommutative, $C_A=A$ and $coPrim(A)_{ab}=coPrim(A)$. So $\iota_A$ is the identity of $coPrim(A)$.\\

Let us now assume that $A$ is commutative. Then $A^*$ is cocommutative, so $\iota_{A^*}$ and $\pi_{A^*}$ are bijective. Hence, their transposes
$\pi_A$ and $\iota_A$ are bijective. \end{proof}

\begin{prop} \label{p10}
($char(K)=0$). Let $A$ be any graded, connected Hopf algebra. Then $Prim(A)\cap A^{+2}=Prim(A)\cap I_A$ and $C_A=\U(Prim(A))$. 
Moreover, the following assertions are equivalent:
\begin{enumerate}
\item $\pi_A$ is injective in degree $n$.
\item $\iota_{A^*}$ is surjective in degree $n$.
\item $\iota_A$ is injective in degree $n$.
\item $\pi_{A^*}$ is surjective in degree $n$.
\item $(Prim(A)\cap A^{+2})_n=[Prim(A),Prim(A)]_n$.
\item $(C_A \cap A^{+2})_n=(C_A^{+2})_n$.
\end{enumerate}
\end{prop}

\begin{proof} As $I_A \cap A^{+2}$, $Prim(A)\cap I_A \subseteq Prim(A)\cap A^{+2}$. Let $x \in Prim(A)\cap A^{+2}$.
Then $x+I_A \in Prim(A_{ab})\cap A_{ab}^{+2}$. As $A_{ab}$ is commutative, $\iota_{A_{ab}}$ is a bijection, so
$C_{A_{ab}} \cap A_{ab}^{+2}=C_{A_{ab}}^{+2}$. As $x+I_A$ is primitive, it belongs to $C_{A_{ab}}$, so $x+I_A \in C_{A_{ab}}^{+2} \cap Prim(A_{ab})$.
By the Cartier-Quillen-Milnor-Moore theorem, $C_{A_{ab}}$ is a  symmetric Hopf algebra, so $C_{A_{ab}}^{+2} \cap Prim(A_{ab})=(0)$.
So $x\in Prim(A)\cap I_A$. 

As $K+Prim(A)$ is a cocommutative subcoalgebra of $A$, it is included in $C_A$. So $Prim(C_A)=Prim(A)$.
By the Cartier-Quillen-Milnor-Moore theorem, $C_A=\U(Prim(A))$.\\

By lemma \ref{8}, 1, 2 and 5 are equivalent, and 3, 4 and 6 are equivalent.
As $C_A=\U(Prim(A))$, $C_A^+=Prim(A)+C_A^{+2}$, so $coPrim(C_A)=\frac{Prim(A)}{Prim(A)\cap C_A^{+2}}$. 
As a consequence, the kernel of $\iota_A$ is $\frac{Prim(A) \cap A^{+2}}{Prim(A) \cap C_A^{+2}}$.
Moreover, as $C_A$ is cocommutative, $\pi_A$ is bijective, so $Prim(A) \cap C_A^{+2}=Prim(A) \cap I_{C_A}=[Prim(A),Prim(A)]$.
Finally, the kernel of $\iota_A$ is $\frac{Prim(A)\cap A^{+2}}{[Prim(A),Prim(A)]}$. So 3 and 5 are equivalent. \end{proof}\\

{\bf Remark.} If the characteristic of the base field is a prime integer $p$, it may be false that $C_A=\U(Prim(A))$.
The weaker assertion telling that $C_A$ is generated by $Prim(A)$, so is a quotient of $\U(Prim(A))$, may also be false.
For example, let us consider the Hopf algebra of divided powers $A$, that is to say the graded dual of $K[X]$.
This Hopf algebra as a basis $(x^{(n)})_{n \geq 0}$,  the product and coproduct are given by:
$$x^{(i)}x^{(j)}=\binom{i+j}{i}x^{(i+j)}, \: \Delta(x^{(n)})=\sum_{i+j=n} x^{(i)}\otimes x^{(j)}.$$
As $A$ is cocommutative, $C_A=A$. It is not difficult to see that $Prim(A)=Vect\left(x^{(1)}\right)$, and the subalgebra generated by $Prim(A)$
is $Vect(x^{(i)}, 0\leq i <p)$, so is not equal to $C_A$.

\subsection{Primitive elements of a free and cofree Hopf algebra}

\begin{prop} \label{11}
($char(K)=0$). Let $H$ be a graded, connected, free and cofree Hopf algebra. Then $\iota_H$ and $\pi_H$ are isomorphisms.
\end{prop}

\begin{proof}  Let us prove by induction on $n$ that $(C_H\cap H^{+2})_n=(C_H^{+2})_n$.
There is nothing to prove for $n=0$. Let us assume the result at all rank $k<n$.

As $H$ is free, let us choose a graded subspace $V$ of $H$, such that $H=T(V)$ as an algebra. 
Then $H^+=V\oplus H^{+2}$, so $coPrim(H)$ is identified with $V$. 
For all $k \in \mathbb{N}$, we denote by $\pi_k$ the projection on $V^k$ in $H=T(V)$. Then the Lie cobracket of $V$ is given by
$\delta(v)=(\pi_1\otimes \pi_1)\circ(\Delta-\Delta^{op})(v)$.

For any word $w=v_1\ldots v_k$ in homogeneous elements of $V$, we put $d(w)=(deg(v_1),\ldots, deg(v_k))$. 
We obtain in this way a gradation of $H$, indexed by words in nonnegative integers. These words are totally ordered in the following way:
if $w=(a_1,\ldots,a_m)$ and $w'=(a'_1,\ldots,a'_n)$ are two different words, then $w<w'$ if, and only if, ($m<n$), or
($m=n$ and there exists an index $i$, such that $a_1=a'_1,\ldots,a_i=a'_i$, $a_{i+1}<a'_{i+1}$).

Let $x\in (C_H\cap H^{+2})_n$. We can write:
$$x=\sum_{k=2}^n \underbrace{\sum_{\substack{(a_1,\ldots,a_k)\in \mathbb{N}^k\\a_1+\ldots+a_k=n}} x_{a_1,\ldots,a_k}}_{x_k},$$
where $x_{a_1,\ldots,a_k}$ is a linear span of words $w$ such that $d(w)=(a_1,\ldots,a_k)$.
Let us put $I_n$ the set of words $(a_1,\ldots,a_k)$ such that $a_1+\ldots+a_k=n$. This set is finite and totally ordered.
Its greatest element is $(1,\ldots,1)$. Let us proceed by a decreasing induction on the smallest $w=(a_1,\ldots,a_n) \in I_n$ such that $x_w \neq 0$.
If $w=(1,\ldots,1)$, then $x$ is in the subalgebra generated by $H_1$. As $H$ is connected, $H_1 \subseteq C_H$, so $x\in C_H^{+2}$.
Let us assume the result for all $w'>w=(a_1,\ldots,a_k)$ in $I_n$. We first prove that $x_{a_1,\ldots,a_k} \in V_{ab}^k$. We put:
$$x_k=\sum_i v_1^{(i)} \ldots v_k^{(i)}.$$
By minimality of $(a_1,\ldots,a_k)$, we obtain:
$$(\pi_k \otimes \pi_1)\circ (\Delta-\Delta^{op})(x)=\sum_i \sum_{j=1}^k v_1^{(i)} \ldots \left(v_j^{(i)}\right)' \ldots v_k^{(j)} \otimes \left(v_j^{(i)}\right)'',$$
with the notation $\delta(v)=v'\otimes v''$ for all $v\in V$. As $x\in C_H$, $\Delta(x)=\Delta^{op}(x)$, so this is zero. As $H$ is freely generated by $V$, we deduce:
$$\sum_i \sum_{j=1}^k v_1^{(i)} \otimes \ldots \otimes \left(v_j^{(i)}\right)'\otimes \ldots \otimes v_k^{(j)} \otimes \left(v_j^{(i)}\right)''=0.$$
Considering the terms of this sum of the form $v_1\otimes \ldots \otimes v_{k+1}$, with $deg(v_1)+deg(v_{k+1})$ minimal, we obtain that:
$$\sum_{\substack{(b_1,\ldots,b_k)\in I_n\\b_1=a_1}} x_{b_1,\ldots,b_n} \in V_{ab}V^{k-1}.$$
Considering the terms of the form $v_1\otimes \ldots \otimes v_{k+1}$, with $deg(v_1)=a_1$ and $deg(v_2)+deg(v_{k+1})$ minimal, we obtain that:
$$\sum_{\substack{(b_1,\ldots,b_k)\in I_n\\b_1=a_1,\:b_2=a_2}} x_{b_1,\ldots,b_n} \in V_{ab}^2V^{k-2}.$$
Iterating the process, we finally obtain that $x_{a_1,\ldots,a_k} \in V_{ab}^k$.\\

We now put:
$$x_{a_1,\ldots,a_k}=\sum_i w_1^{(i)} \ldots w_k^{(i)},$$
where the $w_j^{(i)}$ belong to $(V_{ab})_{a_i}$. As $k\geq 2$, $a_1,\ldots,a_k<n$. By the induction hypothesis, $\iota_H$ is bijective in degree $a_i$ for all $i$.
So there exists $x_j^{(i)}$ in $Prim(H)_{a_j}$, such that $x_j^{(i)}-w_j^{(i)} \in H^{+2}$. 
 We consider the following element:
$$y=x-\underbrace{\sum_i x_1^{(i)} \ldots x_k^{(i)}}_{in (C_H^{+2})_n}.$$
As the $x_j^{(i)}$ are primitive, $y\in (C_H\cap H^{+2})_n$. Moreover, $y_{a_1,\ldots,a_k}=0$ by definition of the $x_j^{(i)}$.
If $y_{b_1,\ldots,b_l} \neq 0$, then $x_{b_1,\ldots,b_l}\neq 0$ or $l>k$. By definition of the order on the words, the smallest $(b_1,\ldots,b_l)$
such that $y_{b_1,\ldots,b_l}\neq 0$ is strictly greater that $(a_1,\ldots,a_k)$: as a consequence, $y\in (C_H^{+2})_n$. So $x \in (C_H^{+2})_n$.\\

{\it Conclusion.} So assertion 6 of proposition \ref{p10} is satisfied for all $n$.
Hence, for all $n$, $\iota_H$ is injective in degree $n$ and $\iota_{H^*}$ is surjective in degree $n$.
By corollary \ref{5}, $H^*$ is isomorphic to $H$. So $\iota_H$ is surjective in degree $n$. A similar proof holds for $\pi_H$. \end{proof}

\begin{cor}\label{12}
($char(K)=0$). Let $H$  be a graded, connected, free and cofree Hopf algebra. Then $[\g,\g]=\g\cap H^{+2}$. 
Let $\h$ be a graded subspace of $H$ such that $\g=[\g,\g]\oplus \h$. 
\begin{enumerate}
\item Then $\h$ freely generates the Lie algebra $\g$. The subalgebra generated by $\h$ is a free Hopf subalgebra of $H$, isomorphic to $\U(\g)$.
\item The graded Hopf algebra $H_{ab}$ is isomorphic to the shuffle algebra $coT(\h)$.
\end{enumerate}\end{cor}

\begin{proof} 1. By proposition \ref{p10}, as $\pi_H$ is injective by proposition \ref{11}, $[\g,\g]=\g \cap H^{+2}$. 
As $\g=[\g,\g]+\h$, $\h$ generates the Lie algebra $\g$. Moreover, $\h\cap H^{+2}\subseteq \h \cap \g\cap H^{+2}=\h \cap [\g,\g]=(0)$,
so $\h$ generates a free algebra $K\langle \h \rangle$. As $\h\subseteq \g$, $K\langle \h \rangle$ is clearly a cocommutative Hopf subalgebra of $H$.
As $\h$ generates the Lie algebra $\g$, by the Cartier-Quillen-Milnor-Moore theorem, $K\langle \h \rangle$ is isomorphic to $\U(\g)$.
So the algebra $\U(\g)$ is freely generated by $\h$, which implies that the Lie algebra $\g$ is freely generated by $\h$. \\

2. As $H$ is self-dual, the Hopf algebras $\U(\g)^*=(C_H)^*$, $(H^*)_{ab}$ and $H_{ab}$ are isomorphic.
By the first point, $\U(\g)$ is isomorphic to $T(\h)$, so $H_{ab}$ is isomorphic to $T(\h^*)^*\approx coT(\h)$. \end{proof}\\

{\bf Remarks.} \begin{enumerate}
\item Corollary \ref{12} is false in characteristic $p$. Indeed, $\g^p=Vect(x^p\:\mid x\in \g)\subseteq \g \cap H^{+2}$. 
It is also false that $\g \cap H^{+2}=[\g,\g]+\g^p$. For example, let us consider a free and cofree Hopf algebra $H$ over $K$, 
such that $\frac{\g}{\g \cap H^{+2}}$ is one-dimensional, concentrated in degree $1$. So $H_1$ is one-dimensional, generated by an element $x$, 
which is primitive. It is not difficult to show that $H_i=Vect(x^i)$ if $i<p$. As a consequence, $\g_1=Vect(x)$ and $\g_i=(0)$ if $2\leq i\leq p-1$.
As $H$ is cofree, there exists $y\in H_p$, such that :
$$\Delta(y)=y\otimes 1+\sum_{i=1}^{p-1}\frac{x^i}{i!}\otimes \frac{x^{p-i}}{(p-i)!}+1\otimes y.$$
It is then not difficult to show that $(x^p,y)$ is a basis of $H_p$. 
So $[\g,\g]_{p+1}=(0)$ and $(\g^p)_{p+1}=(0)$. But $xy-yx$ is a non-zero element of $(\g \cap H^{+2})_{p+1}$.
\item If $H$ is a non-commutative Connes-Kreimer Hopf algebra, then $Prim(H)$ is a free brace algebra;
conversely, any free brace algebra is isomorphic to the Lie algebra of a non-commutative Connes-Kreimer Hopf algebra
\cite{Chapoton,Foissy2,Ronco}. One then recovers the result of \cite{Foissy1}, telling that  in characteristic zero, a free brace algebra is a free Lie algebra.
\end{enumerate}

\subsection{Poincaré-Hilbert series of $H$}

Let $H$ be a graded, connected, free and cofree Hopf algebra. We put:
$$R(h)=\sum_{n=0}^\infty dim(H_n)h^n,\hspace{.5cm}
P(h)=\sum_{n=1}^\infty dim(\g_n) h^n,\hspace{.5cm}
S(h)=\sum_{n=1}^\infty dim\left(\left(\frac{\g}{\g \cap H^{+2}}\right)_n\right) h^n.$$
The coefficients of $R(h)$, $P(h)$ and $S(h)$ will be respectively denoted by $r_n$, $p_n$ and $s_n$.

\begin{prop} \label{13}
($char(K)=0$). The following relations between $R(h)$, $P(h)$ and $S(h)$ are satisfied:
\begin{enumerate}
\item $\displaystyle R(h)=\frac{1}{1-P(h)}$ and $\displaystyle P(h)=1-\frac{1}{R(h)}$.
\item $\displaystyle 1-S(h)=\prod_{n=1}^\infty(1-h^n)^{p_n}$.
\end{enumerate}
\end{prop}

\begin{proof} 1. This comes from the cofreeness of  $H$, see lemma \ref{lemmeplus}. \\

2. We use the notations of corollary \ref{12}. Then $\h$ and $\frac{\g}{\g\cap H^{+2}}$ have the same Poincaré-Hilbert series.
As $\U(\g)$ is freely generated by $\h$, the Poincaré-Hilbert series of $\U(\g)$ is $\frac{1}{1-S(h)}$. By the Poincaré-Birkhoff-Witt theorem,
$\displaystyle \frac{1}{1-S(h)}=\prod_{n=1}^\infty \frac{1}{(1-h^n)^{p_n}}$. \end{proof}\\

The first point of proposition \ref{13} allows to compute $r_n$ in function of $p_1,\ldots,p_n$ and $p_n$ in function of $r_1,\ldots,r_n$;
the second point allows to compute $s_n$ in function of $p_1,\ldots,p_n$ and $p_n$ in function of $s_1,\ldots,s_n$.
For example:
$$\begin{array}{|rclc|crcl|}
\hline r_1&=&p_1 &&& p_1&=&r_1\\
r_2&=&p_2+p_1^2 &&& p_2&=&r_2-r_1^2\\
r_3&=&p_3+p_1^3+2p_2p_1 &&& p_3&=&r_3+r_1^3-2r_2r_1\\
&&&&&&&\\
\hline p_1&=&s_1 &&& s_1&=&p_1\\
p_2&=&\displaystyle s_2+\frac{s_1^2}{2}-\frac{s_1}{2} &&& s_2&=&\displaystyle p_2-\frac{p_1^2}{2}+\frac{p_1}{2}\\
p_3&=&\displaystyle s_3-\frac{s_1}{3}+s_1s_2 +\frac{s_1^3}{3} &&& s_3&=&\displaystyle p_3+\frac{p_1}{3}-p_1p_2-\frac{p_1^2}{2}+\frac{p_1^3}{6}\\
&&&&&&&\\
\hline r_1&=&s_1 &&& s_1&=&r_1\\
r_2&=&\displaystyle s_2+\frac{3s_1^2}{2}-\frac{s_1}{2} &&& s_2&=&\displaystyle r_2-\frac{3r_1^2}{2}+\frac{r_1 }{2}\\
r_3&=&\displaystyle s_3-\frac{s_1}{3}+3s_1s_2-s_1^2+\frac{7s_1^3}{3} &&& s_3&=&\displaystyle r_3+\frac{r_1}{3}-3r_1r_2-\frac{r_1^2}{2}+\frac{13r_1^3}{6}\\
&&&&&&&\\ \hline \end{array}$$

{\bf Remark.} These formulas are false if the characteristic of the base field is not zero.
For example, let us take $S(h)=h$. The formulas of proposition \ref{13} gives then that $p_1=1$ and $p_n=0$ if $n\geq 2$, so $P(h)=h$
and finally $R(h)=\frac{1}{1-h}$: as a consequence, $H=K[h]$ as a Hopf algebra. 
If the characteristic of the base field is a prime integer $p$, then $Prim(H)=Vect(X^{p^k}\mid k\in \mathbb{N})$, so $P(h)\neq h$: contradiction.\\

Here are several applications of these formulas, for the Hopf algebras of the introduction:
$$\begin{array}{c|c|c|c|c|c|c|c|c}
&s_1&s_2&s_3&s_4&s_5&s_6&s_7&s_8\\
\hline\H_{LR} \mbox{ or }\YSym \mbox{ or } \H_{NCK}&1&1&1&3&7&24&72&242\\
\hline 2\mbox{-}As(1)&1&1&2&8&31&141&642&3\:070\\
\hline \FQSym\mbox{ or }\H_{ho}&1&1&2&10&55&377&2\:892&25\:007\\
\hline P\Pi \mbox{ or } \NCQSym&1&2&6&39&305&2\:900&31\:460&385\:080\\
\hline \PQSym \mbox{ or }\H_o&1&2&9&80&901&12\:564&206\:476&3\:918\:025\\
\hline \H_{UBP}&1&2&9&86&1\:083&17\:621&353\:420&8\:553\:300\\
\hline \H_{DP}&1&2&12&165&3\:545&116\:621&5\:722\:481&412\:795\:614\\
\hline R\Pi \mbox{ or }S\Pi&1&3&26&467&12\:518&471\:215&23\:728\:881&1\:545\:184\:651
\end{array}$$
We here denote by $2$-$As(1)$ the free $2$-$As$ algebra on one generator.
The third row is sequence A122826 of \cite{Sloane}, whereas the fifth row is sequence A122720.\\

Note that, for all $n\geq 1$, $s_n=S_n(r_1,\ldots,r_n)$  for particular polynomials $S_n(R_1,\ldots,R_n) \in \mathbb{Q}[R_1,\ldots,R_n]$.
We define these polynomials  here:

\begin{defi} \label{14}
We put in the algebra $\mathbb{Q}[R_1,\ldots,R_n,\ldots][[h]]$:
\begin{eqnarray*}
\sum_{n=1}^\infty P_n(R_1,\ldots,R_n)h^n&=&1-\frac{1}{\displaystyle 1+\sum_{n=1}^\infty R_nh^n},\\
\sum_{n=1}^\infty S_n(R_1,\ldots,R_n)h^n&=&1-\prod_{n=1}^\infty (1-h^n)^{P_n(S_1,\ldots,S_n)}\\
&=&1-\prod_{n=1}^\infty \sum_{k=0}^\infty \frac{P_n(R_1,\ldots,R_n)\ldots (P_n(R_1,\ldots,R_n)-k+1)}{k!} h^{nk}.
\end{eqnarray*}\end{defi}

{\bf Examples.}
\begin{eqnarray*}
S_1(R_1)&=&R_1\\
S_2(R_1,R_2)&=&\displaystyle R_2-\frac{3R_1^2}{2}+\frac{R_1 }{2}\\
S_3(R_1,R_2,R_3)&=&\displaystyle R_3+\frac{R_1}{3}-3R_1R_2-\frac{R_1^2}{2}+\frac{13R_1^3}{6}
\end{eqnarray*}

{\bf Remark.} It is not difficult to show that if $r_i \in \mathbb{Z}$ for all $1\leq i \leq n$, then $P_n(r_1,\ldots,r_n)$, $S_n(r_1,\ldots,r_n)\in \mathbb{Z}$.

\subsection{Free and cofree Hopf subalgebras}

\begin{prop} \label{15}

($char(K)=0$). Let $H$ be a graded, connected, free and cofree Hopf algebra, with a non-degenerate symmetric Hopf pairing. 
Let $\overline{\g'}$ be a graded, non-isotropic subspace of $\frac{\g}{\g\cap H^{+2}}$.
There exists a free and cofree Hopf subalgebra $H'$ of $H$, such that $\overline{\g'}=\frac{\g'}{\g'\cap H'^{+2}}$.
\end{prop}

\begin{proof} Let us choose a graded complement $\h$ of $\g \cap H^{+2}$ in $\g$. The canonical projection $\pi$ from $\g$ to $\frac{\g}{\g \cap H^{+2}}$
induces an isometry $\pi$ from $\h$ to $\frac{\g}{\g \cap H^{+2}}$. Let $\h'$ be $\pi^{-1}_{\mid \h} (\overline{\g'})$: $\h'$ is a supspace of $\h$, 
isometric with $\overline{\g}'$. As $\overline{\g}'$ is non-isotropic, $\h'$ is non-isotropic. Let us put $\h''=\h'^{\perp}\cap \h$. 
Then $\h'$ and $\h''$ are graded subspaces of $\h$ and $\h=\h'\oplus \h''$. By corollary \ref{12}, $\h$ freely generates $\g$ as a Lie algebra. 
We then denote by $\g'$ the sub-Lie algebra of $\g$ generated by $\h'$ and by $\g''$ the Lie ideal of $\g$ generated by $\h''$. Then $\g=\g'\oplus \g''$ and:
$$\g \cap H^{+2}=[\g,\g]=[\g',\g']\oplus [\g,\g'']=(\g' \cap H^{+2})\oplus (\g'' \cap H^{+2}).$$

We now construct by induction on $n$ subspaces $m'_n\oplus m''_n=m_n$ and $w'_n \oplus w''_n=w_n$ of $H_n$, as in lemma \ref{3}, such that:
\begin{enumerate}
\item The subalgebra $H'_{\langle n\rangle}$ generated by $\h'_1\oplus \ldots \oplus \h'_n \oplus w'_1\oplus \ldots \oplus w''_n$ is a Hopf subalgebra.
\item Let $I_{\langle n\rangle}$ be the ideal of $H$ generated by $\h''_1\oplus \ldots \oplus \h''_n\oplus w''_1\oplus \ldots \oplus w''_n$. 
Then $H'_{\langle n\rangle}\perp I_{\langle n\rangle}$.
\end{enumerate}
For $n=0$, all these subspaces are $0$. Let us assume they are constructed at all rank $<n$.
Let us  choose a complement $m'_n$ of $(\g \cap {H'}^{+2}_{\langle n-1 \rangle})_n$ in $({H'}^{+2}_{\langle n-1 \rangle})_n$.
As $H_{\langle n-1 \rangle}$ is freely generated by $\h_1\oplus \ldots \oplus \h_{n-1}\oplus w_1\oplus \ldots \oplus w_n$,
$(H_{\langle n-1 \rangle})_n=H^{+2}_n=(H'_{\langle n-1\rangle})_n\oplus (I_{\langle n-1\rangle})_n$. By the induction hypothesis, these subspaces
are orthogonal. We can then choose a complement $m''_n$ of $m'_n\oplus (\g \cap H^{+2})_n$ in $H^{+2}_n$ included in $(I_{\langle n-1\rangle})_n$.
As $m'_n \subseteq H'_{\langle n-1 \rangle}$, $m'_n \perp m''_n$. We put $m_n=m'_n \oplus m''_n$. We finally choose a $w_n$ as in lemma \ref{3}. \\

As a summary, we obtain a decomposition:
$$H_n=\underbrace{(\g'\cap {H'_{\langle n-1 \rangle}}^{+2})_n \oplus \g''_n}_{(\g \cap H^{+2})_n}
\oplus \underbrace{m'_n \oplus m''_n}_{m_n}\oplus \underbrace{\h'_n \oplus \h''_n}_{\h_n}\oplus w_n,$$
$$(H_{\langle n-1\rangle})_n=\left(\g' \cap {H'_{\langle n-1 \rangle}}^{+2}\right)_n \oplus m'_n.$$
In a basis adapted to the decomposition, by lemma \ref{3} the matrix of the pairing has the form:
$$\left(\begin{array}{cc|cc|cc|cc}
0&0&0&0&0&0&I&0\\\
0&0&0&0&0&0&0&I\\
\hline 0&0&A'&0&0&0&0&0\\
0&0&0&A''&0&0&0&0\\
\hline 0&0&0&0&B'&0&0&0\\
0&0&0&0&0&B''&0&0\\
\hline I&0&0&0&0&0&0&0\\
0&I&0&0&0&0&0&0
\end{array}\right).$$
We naturally decompose $w_n$ as a direct sum $w'_n\oplus w''_n$, in order to split the last column and last row of the matrix.
All the required subspaces are now defined. A basis of $(H'_{\langle n\rangle})_n$ is given by the blocks $1$, $3$, $5$, $7$ of the basis.
A basis of $(I_{\langle n\rangle})_n$ is given by the blocks $2$, $4$, $6$, $8$ of the basis. It is matricially clear that 
$(H'_{\langle n\rangle})_n \perp (I_{\langle n\rangle})_n$.
As these subspaces are in direct sum, $(H'_{\langle n\rangle})_n^\perp=(I_{\langle n\rangle})_n$ in $H_n$.

Let us take $x \in w'_n$. If $y\in (I_{\langle n\rangle})_k$, $z\in H_l$, with $k+l=n$, then as $yz \in (I_{\langle n\rangle})_n$,
$0=\langle x,yz \rangle=\langle \Delta(x) y\otimes z\rangle$. So $\Delta(x) \in (I_{\langle n\rangle}\otimes H)_n^\perp$.
Similarly, $\Delta(x)\in (H\otimes I_{\langle n\rangle})_n^\perp$. As $(I_{\langle n\rangle})_i^\perp=(H'_{\langle n\rangle})_i$
by the preceding remark if $i=n$ and the induction hypothesis if $i<n$, we deduce that $\Delta(x) \in H'_{\langle n\rangle} \otimes H'_{\langle n\rangle}$.
As $H'_{\langle n-1\rangle}$ is a Hopf subalgebra, we deduce that $H'_{\langle n\rangle}$ is a Hopf subalgebra.\\

Let us prove that $(H'_{\langle n\rangle})_i \perp (I_{\langle n\rangle})_i$ for all $i$. We already prove it for $i\leq n$.
If $i>n+1$, let us take $x\in (H'_{\langle n\rangle})_i$ and $y\in (I_{\langle n\rangle})_i$. As $i>n$, we can assume that $y=y_1y_2y_3$,
with $y_1,y_3 \in H$, $y_2 \in (I_{\langle n\rangle})_k$ with $k\leq n$. Using the homogeneity of the pairing and the property of orthogonality at rank $k$:
$$\langle x,y \rangle=\langle \underbrace{\Delta^{(2)}(x)}_{\in H_{\langle n\rangle}^{\otimes 3}}, y_1\otimes y_2 \otimes y_3\rangle=0,$$
as $y_2 \perp H_{\langle n\rangle}$. This ends the induction.\\

{\it Conclusion.} We the take $H'$ the subalgebra generated by $\h'\oplus w'$ and $I$ the ideal generated by $\h''+w''$.
By construction, $H'$ is a free Hopf algebra. As $H$ is freely generated by $\h' \oplus w'\oplus \h'' \oplus w''$, $H'\oplus I=H$.
Moreover $H' \perp I$, so $H'=I^\perp$, comparing their formal series. As a conclusion, $H'$ is a non-isotropic subspace of $H$,
so has a non-degenerate, symmetric Hopf pairing. By construction of $\h'$, $\frac{\g'}{\g'\cap {H'}^{+2}}$ is $\overline{\g'}$. \end{proof}

\subsection{Existence of free and cofree Hopf algebra with a given formal series}

\begin{cor} \label{16}
($char(K)=0$). Let $V$ be a graded space such that $V_0=(0)$, with a symmetric, homogeneous non-degenerate pairing.
There exists a graded, connected, free and cofree Hopf algebra $H$,  with a symmetric, non-degenerate Hopf pairing, 
such that $\frac{\g}{\g \cap H^{+2}}$ is isometric with $V$. Moreover, $H$ is unique, up to an isomorphism.
\end{cor}

\begin{proof} {\it Existence}. Let us choose a basis $(v_d)_{d\in \D}$ formed by homogeneous elements of $V$, where $\D$ is a graded set, 
such that $deg(d)=deg(v_d)$  for all $d\in \D$. Let us consider the Hopf algebra $H_{PR}^\D$ of plane trees  decorated by $\D$ decorated by $\D$: 
from \cite{Foissy1}, it is free and cofree. Moreover, the plane trees $\tdun{$d$}$, $d\in \D$, are linearly independant, primitive elements of $H_{PR}^\D$,
and the space generated by these elements intersects $(H_{PR}^\D)^{+2}$ on $(0)$. So the elements $(\overline{\tdun{$d$}})_{d\in \D}$ are linearly
independant elements of $\frac{\g}{\g \cap (H_{PR}^\D)^{+2}}$. The subspace of $\frac{\g}{\g \cap (H_{PR}^\D)^{+2}}$ generated by 
the $\overline{\tdun{$d$}}$ is identified with $V$. The pairing of $V$ can be arbitrarily extended to $\frac{\g}{\g \cap (H_{PR}^\D)^{+2}}$ 
in a non-degenerate symmetric, homogeneous pairing. This pairing induces a pairing on $H_{PR}^\D$ by theorem \ref{4}. 
From proposition \ref{15}, $H_{PR}^\D$ contains a graded, connected, free and cofree Hopf subalgebra $H'$, such that 
$\frac{\g'}{\g'\cap H'^{+2}}=V$ as a graded quadratic space.\\

{\it Unicity}. Comes directly from proposition \ref{6}. \end{proof}

\begin{cor} \label{17}
($char(K)=0$). Let $(s_n)_{n\geq 1}$ be any sequence of integers. There exists a graded, connected, free and cofree, Hopf algebra $H$ such that
$dim\left(\left(\frac{\g}{\g \cap H^{+2}}\right)_n\right)=s_n$ for all $n \geq 1$. Moreover, it is unique, up to an isomorphism.
\end{cor}

\begin{proof} {\it Existence.} Let $V$ be a graded space such that $dim(V_n)=s_n$ for all $n \geq 1$.
Let us choose a non-degenerate, homogeneous symmetric pairing on $V$. Then corollary \ref{16} proves the existence of $H$.\\

{\it Unicity.} Comes directly from theorem \ref{7}, $2\Longrightarrow 1$. \end{proof}\\

So graded, connected, free and cofree Hopf algebras are entirely determined by sequences of dimensions:

\begin{cor}\label{18}
($char(K)=0$). Let $(r_n)_{n\geq 1}$ be any sequence of integers. There exists a graded, connected, free and cofree Hopf algebra $H$
such that $dim(H_n)=r_n$ for all $n\geq 1$ if, and only if, $S_n(r_1,\ldots,r_n)\geq 0$ for all $n\geq 1$ (recall that $S_n(R_1,\ldots, R_n)$
is defined in definition \ref{14}).
\end{cor}

\begin{proof} $\Longrightarrow$. As $S_n(r_1,\ldots,r_n)=dim\left(\left(\frac{\g}{\g \cap H^{+2}}\right)_n\right)$ for all $n \geq 1$.\\

$\Longleftarrow$. Let us put $s_n=S_n(r_1,\ldots,r_n)$. From corollary \ref{17}, there exists a graded, connected, free and cofree Hopf algebra $H$ 
such that $dim\left(\left(\frac{\g}{\g \cap H^{+2}}\right)_n\right)=s_n$ for all $n \geq 1$. From proposition \ref{13}, $dim(H_n)=r_n$ for all $n \geq 1$. \end{proof}\\

{\bf Remark.} It is not difficult to show that $S_n(R_1,\ldots,R_n)-R_n \in \mathbb{Q}[R_1,\ldots,R_{n-1}]$. 
So the condition on the $r_n$ of corollary \ref{18} can be seen as a growth condition on the coefficients $r_n$.

\subsection{Noncommutative Connes-Kreimer Hopf algebras}

When is a free and cofree Hopf algebra a noncommutative Connes-Kreimer Hopf algebra?

\begin{defi}\textnormal{
The family of polynomials $D_n(R_1,\ldots,R_n) \in \mathbb{Q}[R_1,\ldots,R_n]$ is defined by:
$$\sum_{n=1}^\infty D_n(R_1,\ldots,R_n) h^n=\sum_{n=1}^\infty (-1)^{n+1}n \left(\sum_{k=1}^\infty R_k h^k\right)^n
=\frac{\displaystyle \sum_{k=1}^\infty R_k h^k-1}{\displaystyle \left(\sum_{k=1}^\infty R_k h^k\right)^2}.$$
}\end{defi}

{\bf Examples.}
\begin{eqnarray*}
D_1(R_1)&=&R_1\\
D_2(R_1,R_2)&=&R_2-2R_1^2\\
D_3(R_1,R_2,R_3)&=&R_3-4R_2R_1+3R_1^2
\end{eqnarray*}

\begin{theo} \label{20}
Let $H$ be a graded, connected, free and cofree Hopf algebra. We put $r_n=dim(H_n)$ for all $n \geq 1$. 
Then $H$ is isomorphic to a noncommutative Connes-Kreimer Hopf algebra if, and only if, $D_n(r_1,\ldots,r_n) \geq 0$ for all $n\geq 1$.
\end{theo}

\begin{proof} $\Longrightarrow$. We assume that $H$ is isomorphic to the Hopf algebra of planar trees decorated by $\D$,
here denoted by $\H_{NCK}^\D$. Let $D(h)$ be the formal series of $\D$.
Then, from \cite{Foissy1},  $D(h)=\frac{R(h)-1}{R(h)^2}$, so $Card(\D_n)=D_n(r_1,\ldots,r_n)\geq 0$ for all $n\geq 1$.\\

$\Longleftarrow$. Let $\D$ be a graded set, such that $Card(\D_n)=D_n(r_1,\ldots,r_n)$ for all $n\geq 1$ (it is clear that $D_n(r_1,\ldots,r_n)$ is an integer).
The formal series of $\D$ is $D(h)=\frac{R(h)-1}{R(h)^2}$, so the formal series of $\H_{NCK}^\D$ is:
$$\frac{1-\sqrt{1-4D(h)}}{2D(h)}=R(h).$$
By theorem \ref{7}, $H$ and $\H_{NCK}^\D$ are isomorphic. \end{proof}\\

{\bf Remarks.} \begin{enumerate}
\item Let $H$ be a free and cofree Hopf algebra, such that $s_1=s_2=1$, $s_3=0$. By corollary \ref{17}, this exists.
As $s_2 \neq 0$, $H$ is not equal to $K[X]$. By proposition \ref{13}, $r_1=1$, $r_2=2$, $r_3=4$. So $D_3(r_1,r_2,r_3)=-1<0$:
$H$ is not isomorphic to a non-commutative Connes-Kreimer Hopf algebra.
\item However, at the exception of $K[X]$, all the Hopf algebras of the introduction are isomorphic to a non-commutative Connes-Kreimer Hopf algebra. 
Here are examples of $d_n=Card(\D_n)$ for these objects.
$$\begin{array}{c|c|c|c|c|c|c|c|c}
&d_1&d_2&d_3&d_4&d_5&d_6&d_7&d_8\\
\hline\H_{LR} \mbox{ or }\YSym \mbox{ or } \H_{NCK}&1&0&0&0&0&0&0&0\\
\hline 2\mbox{-}As(1)&1& 0&1&4&17&76&353&1\:688\\
\hline \FQSym\mbox{ or }\H_{ho}&1&0&1&6&39&284&2\:305&20\:682\\
\hline P\Pi \mbox{ or }\NCQSym&1&1&4&28&240&2\:384&26\:832&337\:168\\
\hline \PQSym \mbox{ or }\H_o&1&1&7&66&786&11\:278&189\:391&364\:8711\\
\hline \H_{UBP}&1&1&7&72&962&16\:135&330\:624&8\:117\:752\\
\hline \H_{DP}&1&1&10&148&3\:336&112\:376&5\:591\:196&406\:621\:996\\
\hline R\Pi \mbox{ or }S\Pi&1&2&23&432&11\:929&456\:054&23\:186\:987&1\:518\:898\:380
\end{array}$$ 
The third line is sequence A122827 of \cite{Sloane}, whereas the fifth line is sequence A122705.
\end{enumerate}

\section{Isomorphisms of free and cofree Hopf algebras}

\subsection{Bigraduation of a free and cofree graded Hopf algebra}

Let $V$ be a vector space. A {\it bigraduation} of $V$ is a $\mathbb{N}^2$-graduation of $V$.
If $\displaystyle V=\bigoplus_{(i,j)} V_{i,j}$ is a bigraded space, the {\it first induced graduation} of $V$ is 
$\displaystyle \left(\sum_{j\geq 0} V_{i,j}\right)_{i\geq 0}$ and the {\it second induced graduation} of $V$ is 
$\displaystyle \left(\sum_{i\geq 0} V_{i,j}\right)_{j\geq 0}$. 

There is an immediate notion of bigraded Hopf algebra. A bigraded Hopf algebra $H$ is {\it connected}
if $H_{0,0}=K$ and $H_{0,j}=H_{i,0}=(0)$ for all $i,j\geq 1$. If $H$ is a connected, bigraded Hopf algebra,
then both induced graduation of $H$ are connected.

\begin{lemma}\label{21}
($Char(K)=(0)$). Let $H$ be a graded, connected, free and cofree Hopf algebra.
Consider any bigradation of $\frac{\g}{[\g,\g]}$ such that:
\begin{enumerate}
\item The first induced graduation on $\frac{\g}{[\g,\g]}$ by this bigraduation is the graduation induced by the graduation of $H$.
\item For all $i,j\geq 0$, $\left(\frac{\g}{[\g,\g]}\right)_{i,0}=\left(\frac{\g}{[\g,\g]}\right)_{0,j}=(0)$.
\end{enumerate}
Then there exists a connected bigraduation of $H$, inducing this bigraduation on $\frac{\g}{[\g,\g]}$.
\end{lemma}

\begin{proof} We choose a non-degenerate, homogeneous Hopf pairing on $H$
Let us fix a decomposition $H=[\g,\g] \oplus m\oplus \h \oplus w$ of lemma \ref{3}. Then $\h$ is identified with $\frac{\g}{[\g,\g]}$ via the canonical projection.
We then define a bigraduation on $\h$, making the canonical projection bihomogenous. It is clear that the first induced graduation on $\h$ induced by this bigraduation
is the graduation of $\h$. As $\g$ is freely generated by $\h$, the bigraduation of $\h$ is extended to a graduation of the Lie algebra $\g=[\g,\g]\oplus \h$.
As $\h_{i,0}=\h_{0,j}=0$ for all $i,j$, $\g_{i,0}=\g_{0,j}=(0)$ for all $i,j$.\\

We define $H_{m,n}$ by induction on $m$ such that:
\begin{enumerate}
\item $\g_n$ is a bigraded subspace of  $H_n$ and this bigraduation is the same as the one defined just before.
\item For all $i,j,k,l$ such that $i+k=m$, $H_{i,j}H_{k,l} \subseteq H_{m,j+l}$.
\item For all $x \in H_{m,n}$:
$$\Delta(x)\in \sum_{i+k=m,j+l=n} H_{i,j}\otimes H_{k,l}.$$
\end{enumerate}
For $m=0$, it is enough to take $H_{0,0}=K$ and $H_{0,n}=(0)$ if $n \geq 1$.
Let us assume the result at all ranks $<m$. As $H$ is free, the bigraduation of $H_0\oplus \ldots \oplus H_{m-1}$ can be uniquely extended
to $H^{+2}_m=[\g,\g]_m \oplus m_m$ such that condition $2$ is satisfied. Moreover, this clearly extends the bigraduation of $\g$ and,
for all $x\in [\g,\g]_m \oplus m_m \oplus \h_m$, the third point is satisfied, using the induction hypothesis for $y$ and $z$ and the second point 
if $x=yz\in H^{+2}_m$. 
It remains to define the bigraduation on $w_m$, such that the third point is satisfied for all $x \in w_m$.
Let us choose a basis $(x_i)_{i\in I}$ of $[\g,\g]_m$ made of bihomogeneous elements. Let $(t_i)_{i\in I}$ be the dual basis (for the pairing of $H$)
of $w_n$. We give a bigraduation on $w_n$, putting $t_i$ bihomogeneous of the same bidegree as $x_i$ for all $i$.
Then, for all $i\in I$, if $y$ and $z$ are such that $yz$ is bihomogeneous of a different bidegree, $\langle \Delta(x),y\otimes z\rangle=\langle x,yz\rangle=0$.
So $\Delta(x)$ is bihomogeneous of the same bidegree as $x$: the third point is satisfied. \end{proof}

\subsection{Isomorphisms of free and cofree Hopf algebras}

\begin{prop}
($Char(K)=0$). Let $H$ and $H'$ be two graded, connected, free and cofree Hopf algebra. Thery are isomorphic as (non-graded) Hopf algebras if, and only if,
$dim\left(\frac{\g}{[\g,\g]}\right)=dim\left(\frac{\g'}{[\g',\g']}\right)$.
\end{prop}

\begin{proof} $\Longrightarrow$. Immediate. $\Longleftarrow$. Let us put $n=dim\left(\frac{\g}{[\g,\g]}\right)=dim\left(\frac{\g'}{[\g',\g']}\right) \in \mathbb{N}\cup{\infty}$.
Let us choose a basis $(e_i)_{1\leq i\leq n}$ be a basis of $\frac{\g}{[\g,\g]}$ made of homogeneous elements. We give $\frac{\g}{[\g,\g]}$ a bigraduation,
putting $e_i$ bihomogenous of degree $(deg(e_i),i)$. By lemma \ref{21}, this bigraduation is extended to $H$. Considering the second induced graduation,
$H$ becomes a graded, connected free and cofree Hopf algebra, such that the Poincaré-Hilbert series of $\frac{\g}{[\g,\g]}$ is $\frac{h-h^{n+1}}{1-h}$,
with the convention $h^\infty=0$. The same can be done with $H'$. By theorem \ref{7}, $H$ and $H'$ are isomorphic as graded 
(for the second induced graduation) Hopf algebras. \end{proof}\\

{\bf Remark.} As $\g$ is a free Lie algebra, $dim\left(\frac{\g}{[\g,\g]}\right)$ can be interpreted as the number of generators of $\g$.

\begin{cor}
The Hopf algebras $\FQSym$, $\PQSym$, $\NCQSym$, $R\Pi$, $S\Pi$, $\H_{LR}$, $\YSym$, $\H_{NCK}$ and its decorated version $\H_{NCK}^\D$ 
for any nonempty graded set $\D$, $\H_{DP}$, $\H_o$, $\H_{ho}$, the free $2$-$As$ algebras, and $\H_{UBP}$ are isomorphic.
\end{cor}

\begin{proof} They are all graded, connected, free and cofree, with an infinite-dimensional $\frac{\g}{[\g,\g]}$. \end{proof}

\bibliographystyle{amsplain}
\bibliography{biblio}

\end{document}